\documentclass[a4paper,12pt,reqno]{amsart} 
\usepackage{amssymb, cite} 

\newtheorem*{theorem*}{Theorem} 
\newtheorem{lemma}{Lemma} 
\newtheorem{corollary}{Corollary} 
 
\theoremstyle{definition}

\def\supp{\operatorname{{supp}}}

\begin{document} 
\title[Nonexistence of Extremally Disconnected Groups]
{Nonexistence of Countable\\ Extremally Disconnected Groups\\ with Many Open Subgroups}

\author{Olga V. Sipacheva}

\begin{abstract} 
It is proved that the existence of a countable extremally disconnected Boolean topological 
group containing a family of open subgroups  
whose intersection has empty interior implies the existence of a rapid ultrafilter.
\end{abstract} 

\address 
{Department of General Topology and Geometry, 
Mechanics and Mathematics Faculty, 
Moscow State University}

\email{o-sipa@yandex.ru}

\maketitle

\section{Preliminaries}

The problem of the existence in ZFC of a nondiscrete Hausdorff extremally disconnected topological 
group was posed by Arhangel'skii in 1967 and has been extensively studied since then. Several 
consistent examples have been constructed \cite{Sirota, Louveau, Malykhin, Malykhin2, Zel1, Zel2}, 
but Arhangel'skii's problem remains unsolved. 

All (consistent) examples of extremally disconnected groups known to the author have a base of 
neighborhoods of the identity element consisting of subgroups. The main result of this paper is 
that a countable ZFC example cannot be constructed in this way; to be more precise, we  prove that 
the existence of a countable extremally disconnected Boolean topological 
group containing a family of open subgroups  
whose intersection has empty interior (or, equivalently, admitting a continuous isomorphism onto 
the direct sum $\bigoplus_\omega\mathbb Z_2$ of countably many copies of $\mathbb Z_2$ 
with the product topology) implies the existence of a rapid ultrafilter. Thus, if 
there exists in ZFC a countable nondiscrete extremally disconnected group, then there must exist 
such a group without open subgroups. 

Recall that a topological space is said to be 
\emph{extremally disconnected} if the closure of any open 
set in this space is open (or, equivalently, 
the closures of two disjoint open sets are disjoint). 
Malykhin~\cite{Malykhin} showed that any extremally 
disconnected topological group must contain an open (and 
therefore closed) Boolean subgroup (i.e., a subgroup 
consisting of elements of order~2). Thus, in studying the existence 
of extremally  disconnected groups, it suffices to consider Boolean groups.

Any countable Boolean group is a countable-dimensional vector space over the field $\mathbb 
Z_2$ and, therefore, can be represented as the set of finitely supported maps $g\colon \omega \to 
\mathbb Z_2$ (i.e., is isomorphic to the direct sum $\bigoplus_\omega\mathbb Z_2$ of  
countably many copies of $\mathbb Z_2$). Any 
such isomorphic representation $\varphi\colon G\to \bigoplus_\omega\mathbb Z_2$ 
is determined by the choice 
of a basis $E=\{e_n: n\in\omega\}$ in $G$ for which 
$\varphi (e_i)=(\underbrace{0, \dots, 0}_{\text{$i$ 
times}}, 1, 0, 0\dots)$. We assume $\bigoplus_\omega\mathbb Z_2$ to be endowed with 
the product topology. Thus, 
each isomorphism $\varphi\colon G\to \bigoplus_\omega\mathbb Z_2$ induces a topology on $G$, 
which we refer to 
as the \emph{product topology induced} by $\varphi$ (or \emph{associated} with the corresponding 
basis $E$).  

Given an isomorphism $\varphi\colon G\to \bigoplus_\omega\mathbb Z_2$, we can treat elements of 
$G\cong \bigoplus_\omega\mathbb Z_2$ as maps to $\mathbb Z_2$ finitely supported on $\omega$, 
so that the maps 
$$ 
\min\colon G\to \omega, \qquad \min g= \min \supp g, 
$$ 
and 
$$ 
\max\colon G\to \omega, \qquad \max g= \max \supp g 
$$
are defined.

We denote the zero element of $G$ by $\mathbf{0}$. 

All topological groups under consideration are assumed to be Hausdorff. 

Given $k, n\in \omega$, we write 
$[k, n]$ for $\{i\in \omega: k\le i\le n\}$. The sign $\bigsqcup$ denotes disjoint union. Thus,  
$S= \bigsqcup_{n\in \omega} S_n$ means that the $S_i$ form a disjoint partition of $S$. 

An ultrafilter $\mathcal U$ on $\omega$ is a \emph{$P$-point ultrafilter} if,  
given any partition $\omega=\bigsqcup_{n\in \omega} A_n$ with $A_n\notin 
\mathcal U$, $n\in \omega$, there exists an $A\in \mathcal {U}$ such that $|A\cap A_n|<\aleph_0$ 
for any $n$. The nonexistence of $P$-point ultrafilters is consistent with ZFC~\cite{Shelah}.

An ultrafilter $\mathcal U$ on $\omega$ is \emph{selective}, or \emph{Ramsey}, if, given any 
partition $\omega=\bigsqcup_{n\in \omega} A_n$ with $A_n\notin 
\mathcal U$, $n\in \omega$, there exists 
an $A\in \mathcal {U}$ such that $|A\cap A_n|\le 1$ for any $n$.

An ultrafilter $\mathcal U$ on $\omega$ is said to be \emph{rapid} if every function 
$\omega\to\omega$ is majorized by the increasing enumeration of some element of $\mathcal U$.  

Any Ramsey ultrafilter is rapid (and a $P$-point). The 
nonexistence of rapid ultrafilters is consistent with ZFC~\cite{rapid}.

Note that if $X$ and $Y$ are sets, $\mathcal U$ is an ultrafilter on $X$, and $f\colon 
X\to Y$ is a map,  then the family $\mathcal V= \{A \subset Y: f^{-1}(A) \in \mathcal U\}$ is an   
ultrafilter on $f(X)$. We denote it by $f(\mathcal U)$. 

\section{Statements}

The main result of this paper is the following theorem. 

\begin{theorem*}
The existence of a countable extremally disconnected Boolean topological 
group containing a family of open subgroups  
whose intersection has empty interior implies the 
existence of a rapid ultrafilter. 
\end{theorem*}

\begin{corollary}
\label{c0}
The nonexistence of a countable extremally disconnected Boolean topological 
group containing a family of open subgroups  
whose intersection has empty interior is consistent with ZFC. 
\end{corollary}

The proof of the theorem uses the following two lemmas.

 \begin{lemma}
\label{l1}
Suppose that $G$ is a countable Boolean topological group which can be represented 
as $G\cong \bigoplus_\omega\mathbb Z_2$ in such a way that 
\begin{equation}
\label{eq1}
\begin{aligned}
&\text{for any function $f\colon \omega\to\omega$, the set}\\
&U_f = \{x\ne \mathbf{0}: \max x > f(\min x)\} \cup \{\mathbf{0}\}\\
&\text{contains an open neighborhood of $\mathbf{0}$.}
\end{aligned} 
\end{equation}
Let $f\colon \omega\to \omega$ be any increasing function, and let $V$ be a neighborhood of zero 
such that $2V\subset U_g$ for $g\colon \omega\to \omega$ defined by $g(n)=f(2^{n+1})$. Then 
the increasing enumeration $\{m_0, m_1, \dots\}$ of $\max (V\setminus \{\mathbf 0\})$ majorizes $f$, 
i.e., $m_i> f(i)$ for all $i\in \omega$. 
\end{lemma}

\begin{lemma}
\label{l2} 
Let $G$ be a countable nondiscrete  Boolean topological group which contains 
a (countable) decreasing family of open subgroups $G_i \subset G$, $i \in \omega$, with 
$\bigcap G_i = \{\mathbf{0}\}$. Then there exists a continuous isomorphism $\varphi\colon 
G\to \bigoplus_\omega\mathbb Z_2$, i.e., an isomorphism $\varphi$ such that 
the product topology on $G$ induced by $\varphi$ is contained in the given topology. 
\end{lemma}

These lemmas also imply the following assertions. 

\begin{corollary}
\label{c1}
The existence of a countable nondiscrete group $G$ satisfying condition~\eqref{eq1} for 
some representation $G\cong \bigoplus_\omega\mathbb Z_2$ implies the existence of a rapid 
ultrafilter.
\end{corollary}

\begin{proof}
Any nonprincipal ultrafilter $\mathcal U$ on $G$ converging to zero contains 
all neighborhoods of zero. Hence $\max \mathcal U$ contains $\max V$ for any neighborhood $V$ of 
zero and is therefore rapid.
\end{proof}

It is easy to see that, for any countable Boolean group $G$ with a fixed isomorphism  
$\varphi\colon G\to \bigoplus_\omega\mathbb Z_2$ and any map $f\colon \omega\to \omega$, the set 
$$ 
C_f = G\setminus U_f = \{x\in G\setminus\{\mathbf{0}\}: \max x \le f(\min x)\} 
$$ 
is discrete in the product topology induced by $\varphi$, and $C_f\cup \{\mathbf{0}\}$ is 
closed in this topology. This, together with Lemma~\ref{l2}, implies the following corollary, which 
gives a partial answer to Protasov's question on the existence in ZFC of a countable 
nondiscrete topological group in which all discrete subsets are closed (see \cite[Chapter 13, 
Question~16]{ProtQ}. 

\begin{corollary}
\label{c2}
Let $(G, \tau)$ be a countable nondiscrete Boolean topological group. Suppose that 
$G$ has no discrete subsets with a unique limit point and contains 
a family of open subgroups $G_i \subset G$, $i \in \omega$, such that 
$\bigcap G_i = \{\mathbf{0}\}$.  Then there exists a rapid 
ultrafilter. 
\end{corollary}

In the next corollary, by a maximal nondiscrete group topology we mean a 
nondiscrete group topology which is maximal among all nondiscrete group topologies. 

\begin{corollary}
\label{c3}
The following statement is consistent with ZFC: 
Any countable Boolean group with a maximal 
nondiscrete group topology contains a discrete subset with a  unique limit point. 
\end{corollary}

\begin{proof}
This statement is true in any model of ZFC containing no rapid ultrafilters. Indeed, let $G$ be a 
countable Boolean group with a maximal nondiscrete group topology $\tau$. Choose any basis $E$ in 
the vector space $G$. If $\tau$ contains the product topology $\mathbb T$ associated with $E$, 
then we can apply Corollary~\ref{c2}. If $\tau\not\supset  \mathbb T$, then the maximality of 
$\tau$ implies the existence of neighborhoods of zero $U_0$ in $\tau$ and $V_0$ in $\mathbb T$ with 
$U_0\cap V_0=\{\mathbf {0}\}$. For each $f\colon \omega\to \omega$, all sets 
$\{x\in C_f: \min x\le  n\}$ are finite and, hence, 
$$
\{U\cap C_f: \mathbf{0}\in U\in \tau\} \supset \{V\cap C_f: \mathbf{0}\in V\in \mathbb T\}. 
$$
Consequently, $G$ satisfies condition \eqref{eq1} with $V= U_0$, and Corollary~\ref{c1} applies.
\end{proof}

\section{Proofs}

The short proof of Lemma~\ref{l1} presented below was found and kindly communicated 
to the author by E.~A.~Reznichenko.

\begin{proof}[Proof of Lemma~\ref{l1}]
Note that if $X\subset G\setminus\{\mathbf 0\}$ and $|X|>2^k$, then there are $x, y \in X$ for which 
$\min (x+y) > k$: it suffices to take any $x$ and $y$ whose first $k$ coordinates coincide.  
This observation was essentially made in \cite[Proof of Theorem~5.19, Case~2]{ProtQ}. 

Let $\{m_0, m_1, \dots\}$ be the increasing 
enumeration of $\max(V\setminus \{\mathbf 0\})$. For 
each $i\in \omega$, choose $x_i\in V$ with $\max x_i = m_i$. We have 
$m_0>g(\min x_0)> f(\min x_0) \ge f(0)$. Let us show that $m_n> f(n)$ for any $n>0$. Take $k$ 
for which $2^k< n\le 2^{k+1}$.  It follows from the above observation that $\min(x_i + x_j) > 
k$ for some $i<j\le n$. Clearly, $\max x_j=m_j > m_i =\max x_i$ implies 
$\max(x_i + x_j)=\max x_j$. We have 
$$
f(n) \le f(2^{k+1}) = g(k) \le g(\min (x_i+x_j)) < \max (x_i+x_j) = m_j \le m_n.
$$
\end{proof}

\begin{proof}[Proof of Lemma~\ref{l2}]
We treat $G$ as a vector space  over the field 
$\mathbb Z_2$ and 
the $G_i$ as its subspaces. 
Obviously, to prove the lemma, it suffices to construct 
a basis $E=\{e_n:  n\in\omega\}$ 
such that, for 
every $i\in \omega$,  there exists a $J_i\subset \omega$ for which $G_i = 
\langle e_n:n \in J_i\rangle $. Indeed, if $E$ is such a basis, then the assumption 
$\bigcap G_i=\{\mathbf{0}\}$ implies $\bigcap J_i = \varnothing$, and all linear spans $\langle 
e_k: k \ge n\rangle $ (which form a base of neighborhoods of zero in the product topology 
associated with the basis $E$) are open as subgroups with nonempty interior. 

In each (nontrivial) quotient space $G_i/G_{i+1}$, we 
take a basis $\{\varepsilon_\alpha: \gamma\in I_i\}$, where 
$|I_i|= \dim G_i/G_{i+1}$, and let 
$e_\alpha$ be representatives of 
$\varepsilon_\gamma$ in $G_i$. We assume the 
(at most countable) index sets $I_i$ to be well ordered and disjoint,  
let $I=\bigcup_{i\in \omega}I_i$, and endow $I$ with the lexicographic 
order (for $\alpha, \beta\in I$, we say that 
$\alpha<\beta$ if $\alpha\in I_i$, 
$\beta\in I_j$, and either $i<j$ or $i=j$ and $\alpha< \beta$ 
in $I_i$). For any $i\in \omega$ and $\alpha\in I_i$, we define 
$H_\alpha$ to be the subspace of $G$ spanned by $\{e_\beta: \alpha 
\in I_i, \beta\ge \alpha\}$ and $G_{i+1}$. Thus, 
$H_\beta$ is defined for each $\beta\in I$; moreover, if $\beta, \gamma \in I$ 
and $\beta < \gamma$ in $I$, then $H_\beta\supset H_\gamma$, and 
if $\alpha$ is the least element of $I_i$, then $H_\alpha=G_i$. 
This means that the subspaces $H_\alpha$ form a decreasing (with respect 
to the order induced by $I$) chain of subspaces refining the chain 
$G_0\supset G_1\supset\dots$\,.  Note that the lexicographic order on $I$ 
is a well-order, so for each $\alpha\in I$, its immediate successor 
$\alpha+1$ is  defined; by construction, we have 
$\dim H_{\alpha}/H_{\alpha+1}= 1$ for every $\alpha \in I$. 

Clearly, it suffices to construct a basis 
$E'=\{e'_\alpha: \alpha\in I\}$ with the property
$$
e'_\alpha\in H_\alpha\setminus H_{\alpha+1}
 \qquad\mbox{for every}\quad  \alpha\in I; 
\eqno{(\star)}
$$
the required basis $E$ is then obtained by reordering $E'$. Moreover, 
property~$(\star)$ ensures the linear independence of $E'$, so it is sufficient 
to construct a set of vectors spanning $G$ with this property. 

Take any basis $E''= \{e''_n: n\in \omega\}$ in $G$. We construct $E'$ 
by induction on $n$. 

Let $\alpha_0$ be the (unique) element of $I$ for which $e''_0\in 
H_{\alpha_0}\setminus H_{\alpha_0+1}$. We set $e'_{\alpha_0}=e''_0$. 

Suppose that $k$ is a positive integer 
and we have already defined elements $\alpha_i\in I$ 
and  vectors $e'_{\alpha_i}$ for $i<k$ so that 
$e'_{\alpha_i}\in H_{\alpha_i}\setminus H_{\alpha_i+1}$ and 
$\langle e'_{\alpha_0}, \dots, e'_{\alpha_{k-1}}\rangle = 
\langle e''_0,\dots, e''_{k-1}\rangle $. Take the (unique) $\alpha\in I$ 
for which $e''_k\in H_{\alpha}\setminus H_{\alpha+1}$. If $\alpha$ 
is unoccupied, that is, $\alpha \ne \alpha_i$ 
for any $i<k$, then we set $\alpha_k=\alpha$ and $e'_{\alpha_k}=e''_k$. If 
$\alpha$ is already occupied, i.e., $\alpha=\alpha_m$ for some $m<k$, then 
we take the (unique) $\beta\in I$ for which $e''_k+e'_{\alpha_m}\in   
H_{\beta}\setminus H_{\beta+1}$. (Clearly, $\beta>\alpha$, because 
$\dim H_{\alpha}/H_{\alpha+1}=1$ and $e''_k, e'_{\alpha_m} 
\in H_{\alpha}\setminus H_{\alpha+1}$). If $\beta$ is unoccupied, 
then we set $\alpha_k= \beta$ and $e'_{\alpha_k}=e''_k+e'_{\alpha_m}$; if 
$\beta= \alpha_l$ for $l<k$, then we take the (unique) $\gamma \in I$ for 
which $e''_k+e'_{\alpha_m}+e'_{\alpha_l}\in   
H_{\gamma}\setminus H_{\gamma+1}$ (clearly, $\gamma>\beta>\alpha$), and 
so on.  Only finitely many ($k-1$) indices from $I$ are occupied; therefore, 
after finitely many steps, we obtain $e''_k+e'_{\alpha_m}+e'_{\alpha_l}+
\dots + e'_{\alpha_s}\in H_{\delta}\setminus H_{\delta+1}$ for an unoccupied 
index $\delta$. We set $\alpha_k=\delta$ and $e'_{\alpha_k}= 
e''_k+e'_{\alpha_m}+e'_{\alpha_l}+\dots +e'_{\alpha_s}$. 

As a result, we obtain a set of vectors $E'=\{e'_{\alpha_n}:n\in \omega\}$ 
such that  
$e'_{\alpha_n}\in H_{\alpha_n}\setminus H_{\alpha_n+1}$ and 
$\langle e'_{\alpha_0}, \dots, e'_{\alpha_n}\rangle =
\langle e''_0, \dots, e''_n\rangle $ for every $n\in \omega$. The 
latter means that $E'$ spans $G$, because so does the basis 
$E''$. 

Formally, it may happen that not all of the indices $\alpha\in I$ 
are occupied, that is, $\{\alpha_n: n\in\omega\}=J\subsetneq I$. In 
this case, we take arbitrary  vectors $e'_\alpha\in H_\alpha\setminus H_{\alpha+1}$ for $\alpha\in 
I\setminus J$ and put them to $E'$. The set $E'$ thus enlarged 
satisfies condition $(\star)$ and is therefore linearly independent; 
thus, it cannot differ from the initial $E'$, because the latter 
spans $G$, and $J$ in fact coincides with $I$, i.e., 
$E'=\{e'_{\alpha_n}:n\in \omega\}=\{e'_\alpha:\alpha\in I\}$. 
This completes the proof of the lemma.  
\end{proof}

\begin{proof}[Proof of the theorem]
Let $G$ be a countable extremally disconnected  Boolean group, and let  
$G_i \subset G$, $i \in \omega$, be its open (and hence closed) subgroups such that 
the intersection $\bigcap_{i\in \omega} G_i$ has 
empty interior. Then the quotient of $G$ by $\bigcap_{i\in \omega} G_i$ is Hausdorff, nondiscrete,  
and extremally disconnected (being an open image of an extremally disconnected group). 
Thus, we can assume without loss of generality that   $\bigcap_{i\in 
\omega} G_i=\{\mathbf {0}\}$. We can also assume that $G_{i+1}\subset G_i$. According to 
Lemma~\ref{l2}, there exists a continuous isomorphism $\varphi\colon G\to 
\bigoplus_\omega\mathbb Z_2$. 

If $G$ satisfies condition~\eqref{eq1} for $\min$ and $\max$ associated with this isomorphism, 
then the existence of a rapid ultrafilter follows by Lemma~\ref{l1}. Suppose that $G$ does not 
satisfy~\eqref{eq1}, i.e., there exists a function $f\colon \omega\to \omega$ such that 
$\mathbf{0}$ is a limit point of $C_f$, where 
$$
C_f = U_f\setminus \{\mathbf{0}\}=\{x\in G\setminus \{\mathbf{0}\}: \max x\le f(\min x)\}.
$$
Let us enumerate $C_f$ as $\{c_0, c_1, \dots\}$ in such a way that
\begin{equation}
\label{eq3}
\min c_k < \min c_n \implies k< n.
\end{equation}
This can be done, because all sets of the form $\{x\in C_f: \min x = n\}$ are finite. For each $n 
\in \omega$, we set 
$$
U_{c_n}=\{x\in G: \min x > f(\min c_n)\}\cup \{\mathbf{0}\}.
$$
All $U_{c_n}$ are open neighborhoods of zero in the product topology on $G$ induced by $\varphi$ 
(and, therefore, in the topology of $G$). Clearly, 
$$
(c_k + U_{c_k}) \cap (c_n + U_{c_n}) = \varnothing \quad \text {for}\quad k\ne n.
$$
Thus, in Zelenyuk's terminology~\cite{Zelenyuk}, $C_f$ is a strongly discrete set in the extremally 
disconnected group $G$, and according to \cite[Lemma~2]{Zelenyuk}, the subsets of $C_f$ containing 
$\mathbf{0}$ in their closures form an ultrafilter on $C_f$. Let us denote it by $\mathcal U$. In 
the proof of Theorem~3 of \cite{Zelenyuk} Zelenyuk showed that this ultrafilter is partially 
selective. This means that there exists a map $u\colon C_f \to \mathcal{U}$ such that, given any 
partition $C_f = \bigsqcup_{i\in \omega}A_i$ with $A_i \notin \mathcal{U}$, there exists 
an $A\in \mathcal{U}$ for which $A\cap A_i\cap u(x)=\varnothing$ whenever $x\in A\cap A_i$. 
(Formally, Zelenyuk's definition of a partially selective ultrafilter 
differs from that given above, but it is easy to see that these definitions are equivalent.) 
Substituting $D=C_f$ in Zelenyuk's proof of his Theorem~3 in~\cite{Zelenyuk} (no other 
substitutions or changes are needed), we see that, for any partition 
$C_f = \bigsqcup_{i\in \omega}A_i$ with $A_i \notin \mathcal{U}$, there exists an $A\in 
\mathcal{U}$ such that, for any $i\in \omega$, we have $A\cap A_i \cap U_c=\varnothing $ whenever 
$c\in A\cap A_i$ (in our terminology, this means that the partial selectivity of $\mathcal{U}$ is 
witnessed by the map $c\mapsto U_c$.)  

Note that, by virtue of the condition~\eqref{eq3} on the enumeration of $C_f$, we have 
$$
U_{c_n}\cap C_f = \bigl\{c_i: i\ge \min\{j: \min c_j> f(\min c_n)\}\bigr\}
$$
for each $n$. Bearing this in mind and using the natural bijection $C_f\rightleftarrows \omega$ 
defined by $c_n\mapsto n$, we can summarize the preceding paragraph as follows. To each $n\in 
\omega$ we can assign an $m_n \in \omega$ so that there exists an ultrafilter $\mathcal V$ on 
$\omega$ with the following property: 
\begin{equation}
\label{eq4}
\begin{aligned}
&\text{for any partition  $\omega = \bigsqcup_{i\in \omega}A_i$ with $A_i \notin 
\mathcal{V}$,}\\ 
&\text{there exists an $A\in \mathcal V$ such that, for any $i\in \omega$,}\\
&\text{$n\in A\cap A_i$ implies $A\cap A_i\cap \{j: j \ge m_n\}=\varnothing$.} 
\end{aligned} 
\end{equation} 
This means, in particular, that $\mathcal V$ is a $P$-point ultrafilter. It remains to apply the 
following lemma.

\begin{lemma}
For any ultrafilter $\mathcal{U}$ on $\omega$ satisfying condition~\eqref{eq4} for some 
correspondence $n\mapsto m_n$, there exists a map $\Phi\colon \omega\to \omega$ such that 
$\Phi(\mathcal{U})$ is Ramsey. 
\end{lemma}

\begin{proof}
We set $M_0= m_0$, $M_1= m_{m_0}$, $M_2= m_{m_{m_0}}$, and so on; for positive $n\in \omega$, we set 
$$
M_n= m_{{}_{M_{n-1}}}, \quad\text{i.e.,}\quad
M_n = m_{\underbrace{{}_{m_{m_{\dots_{\rlap{\tiny${}_m$}}}}}}_{\text{$n$ times}}}{\vrule depth 9pt 
width0pt}_{\,\,\tiny{{}_0}}.
$$
Without loss of generality, we can assume that the $m_n$ strictly increase.  
In this case, the $M_n$ strictly increase as well. The required map $\Phi$ can be defined 
as follows: given an integer $i\in \omega$, we find the interval $[M_n, M_{n+1}-1]$ to which it 
belongs and set $\Phi(i) = n$.  The map $\Phi$ thus defined contracts each interval of the form 
$[M_n, M_{n+1}-1]$ to the point $n$ and is surjective.  

Let us show that the ultrafilter $\Phi(\mathcal U)$ is Ramsey. Consider a partition $\omega = 
\bigsqcup_{i\in \omega} B_i$ with $B_i\notin \Phi(\mathcal U)$.  We must choose $B\in \Phi({\mathcal 
U})$ so that $|B\cap B_i|\le 1$ for all $i\in \omega$. Let $A_i = \Phi^{-1}(B_i)$. We have 
$A_i\notin \mathcal U$ and $\omega=\bigsqcup_i A_i$. Take $A\in \mathcal U$ as in~\eqref{eq4}, i.e., 
such that, for any $i\in \omega$, $n\in A\cap A_i$ implies $A\cap A_i\cap \{j: j\ge 
m_n\}=\varnothing$. Since $\mathcal U$ is an ultrafilter, we can choose $A$ so that either 
$$
A\cap \bigcup  \{[M_n, M_{n+1}-1]: \text{$n$ is odd}\}= \varnothing
$$
or 
$$
A\cap \bigcup  \{[M_n, M_{n+1}-1]: \text{$n$ is even}\}= \varnothing.
$$
Suppose for definiteness that $A$ satisfies the former condition. Let us show that $|\Phi(A)\cap 
B_i|\le 1$ for all $i\in \omega$. Consider any nonempty intersection $\Phi(A)\cap B_i$. Let $n$ be 
its least element. We have $n=\Phi(k)$ for some $k\in A\cap A_i$. By assumption, $A\cap A_i\cap 
\{j: j\ge m_k\}=\varnothing$. We have: $k\in [M_n, M_{n+1}-1]$ (because $\Phi(k) = n$), $n$ is 
even (by the choice of $A$), $A\cap [M_{n+1}, M_{n+2}-1]=\varnothing$, and 
$m_k\in [M_{n+1}, M_{n+2}-1]$ (because $k< 
M_{n+1}$ and  $k\ge M_{n}$). Hence $A\cap A_i \cap \{j: j\ge 
M_{n+1}\}=\varnothing$, and 
$$
\Phi(A) \cap \Phi(A_i) \cap \Phi(\{j: j\ge 
M_{n+1}\})=\Phi(A) \cap B_i \cap \{j: j\ge n+1\}=\varnothing.
$$
This means that $\Phi(A) \cap B_i = \{n\}$.
\end{proof}

\end{proof}

\section*{ACKNOWLEDGMENTS}

The author is very grateful to Evgenii Reznichenko, who suggested 
the short proof of the main lemma, which greatly shortened and simplified the exposition, and to 
Anton Klyachko for useful discussions.

\end{document}